\documentclass[12pt]{amsart}
\usepackage{amssymb}

\usepackage{enumerate,color,titletoc}%,mathrsfs}
\usepackage[pagebackref,colorlinks,linkcolor=red,citecolor=blue,urlcolor=blue,hypertexnames=true]{hyperref}
\usepackage{url}
\usepackage{verbatim}

\usepackage{pdfpages,multicol}

\renewcommand{\subset}{\subseteq}
\renewcommand{\emptyset}{\varnothing}

\def\bes{\begin{equation*} }
\def\ees{\end{equation*} }

\def\w{w}

\makeatletter
\newsavebox\myboxA
\newsavebox\myboxB
\newlength\mylenA

\newcommand*\xoverline[2][0.75]{%
    \sbox{\myboxA}{$\m@th#2$}%
    \setbox\myboxB\null% Phantom box
    \ht\myboxB=\ht\myboxA%
    \dp\myboxB=\dp\myboxA%
    \wd\myboxB=#1\wd\myboxA% Scale phantom
    \sbox\myboxB{$\m@th\overline{\copy\myboxB}$}%  Overlined phantom
    \setlength\mylenA{\the\wd\myboxA}%   calc width diff
    \addtolength\mylenA{-\the\wd\myboxB}%
    \ifdim\wd\myboxB<\wd\myboxA%
       \rlap{\hskip 0.5\mylenA\usebox\myboxB}{\usebox\myboxA}%
    \else
        \hskip -0.5\mylenA\rlap{\usebox\myboxA}{\hskip 0.5\mylenA\usebox\myboxB}%
    \fi}
\makeatother

\newtheorem{theorem}            {Theorem}[section]
\newtheorem{thm}           [theorem]{Theorem}
\newtheorem{cor}          [theorem]{Corollary}
\newtheorem{prop}        [theorem]{Proposition}

\newtheorem{lem}              [theorem]{Lemma}
\newtheorem{lemma}              [theorem]{Lemma}

\newtheorem{example}            [theorem]{Example}
\newtheorem{ex}            [theorem]{Example}
\newtheorem{rem}             [theorem]{Remark}
\newtheorem{remark}             [theorem]{Remark}

\newtheorem{algor}         [theorem]{Algorithm}

\numberwithin{equation}{section}

\def\beq{ \begin{equation} }
\def\eeq{ \end{equation} }

\def\bep{\begin{proof}}
\def\eep{\end{proof}}

\def\ben{\begin{enumerate}}
\def\een{\end{enumerate}}

\def\bet{\begin{theorem}}
\def\eet{\end{theorem}}

\def\bel{\begin{lemma}}
\def\eel{\end{lemma}}

\def\la{\langle}

\def\chrisT{\ast}
\def\FNSx{\R \langle x,x^{\chrisT} \rangle}

\def\FA{\R \langle x,x^{\chrisT} \rangle}

\def\x{x}
\def\xs{x^*}
\newcommand{\ax}{\langle\x\rangle}
\newcommand{\axs}{\langle\x,\xs\rangle}
\newcommand{\axstar}{\langle\xs\rangle}

\def\lnss{left Nullstellensatz property}

\def\qrr{real radical}
\def\rad{vanishing radical}
\newcommand{\rr}[1]{\sqrt[\mathrm{re}]{#1}}
\newcommand{\sq}[1]{\sqrt{#1}}

\def\J{\mathcal{I}}
\def\tr{\mbox{tr}}

\def\x{x}
\def\xs{x^*}

\def\cI{ {\mathcal I} }

\newcommand{\N}{{\mathbb N}}

\newcommand{\RR}{{\mathbb R}}
\newcommand{\R}{{\mathbb R}}

\def\la{\lambda}

\newcommand{\df}[1]{{\bf{#1}}{\index{#1}}}
\begin{document}

\title{On real one-sided ideals in a free algebra}

\author[Cimpri\v c]{Jakob Cimpri\v c${}^1$}
\address{Jakob Cimpri\v c, Department of Mathematics, University of 
Ljubljana, Slovenia}
\email{cimpric@fmf.uni-lj.si}
\thanks{${}^1$Research supported by the grant P1--0222 from the 
Slovenian Research Agency, ARRS}

\author[Helton]{J. William Helton${}^2$}
\address{J. William Helton, Department of Mathematics\\
  University of California \\
  San Diego}
\email{helton@math.ucsd.edu}
\thanks{${}^2$Research supported by NSF grants
DMS-0700758, DMS-0757212, and the Ford Motor Co.}
\author[Klep]{Igor Klep${}^{3}$}
\address{Igor Klep, Department of Mathematics, 
The University of Auckland, New Zealand}
\email{igor.klep@auckland.ac.nz}
\thanks{${}^3$Supported by the Faculty Research Development Fund (FRDF) of The
University of Auckland (project no. 3701119). Partially supported by the Slovenian Research Agency grant P1-0222.}
\author[McCullough]{Scott McCullough${}^4$}
\address{Scott McCullough, Department of Mathematics\\
  University of Florida, Gainesville %\\
   % Box 118105\\
   %  Gainesville, FL 32611-8105\\
   %  USA
   }
   \email{sam@math.ufl.edu}
\thanks{${}^4$Research supported by NSF grants DMS-0758306 and DMS-1101137.}

\author[Nelson]{Christopher Nelson${}^5$}
\address{Christopher Nelson, Department of Mathematics\\
  University of California \\
  San Diego}
\email{csnelson@math.ucsd.edu}
\thanks{${}^5$Research supported by NSF grants
DMS-0700758, DMS-0757212}

\date{\today}

\setcounter{tocdepth}{2}
\contentsmargin{2.55em} 
\dottedcontents{section}[3.8em]{}{2.3em}{.4pc} 
\dottedcontents{subsection}[6.1em]{}{3.2em}{.4pc}
%\dottedcontents{subsubsection}[8.4em]{}{4.1em}{.4pc}

\makeatletter
\newcommand{\mycontentsbox}{%
{\centerline{NOT FOR PUBLICATION}
\small\tableofcontents}}
\def\enddoc@text{\ifx\@empty\@translators \else\@settranslators\fi
\ifx\@empty\addresses \else\@setaddresses\fi
\newpage\mycontentsbox}
\makeatother

%%%%%%%%%%%%%%%%%%%%%%%%%%%%

\begin{abstract}  
In real algebraic geometry there are several notions
of the radical of an ideal $I$.
There is the \rad{} $\sq{I}$ defined as the set of all real polynomials vanishing on the real zero set of $I$,
and the real radical
$\sqrt[\rm re]{I}$ defined 
as the smallest real ideal
containing $I$. 
(Neither of them is to be confused with the usual radical from commutative algebra.)
By the real Nullstellensatz, $\sq{I}=\rr{I}$.
This paper 
focuses on
extensions of these 
to  the free algebra  $\RR \langle x,x^\ast \rangle$
 of noncommutative real polynomials  in  $x=(x_1,\ldots,x_g)$ and 
 $x^*=(x_1^*,\ldots,x_g^*)$.

 We work with a natural notion of the (noncommutative real)
 zero set  $V(I)$ of a   left ideal $I$ in 
   $\RR \langle x,x^\ast \rangle$. 
 The \rad{} $\sqrt[]{I}$ of $I$ is   the set of
all $p \in \RR \langle x,x^\ast \rangle$  which vanish on
$V(I)$. 
The earlier paper \cite{CHMN} gives an appropriate
 notion of $\sqrt[\rm re]{I}$ and proves $\sq{I}=\sqrt[\rm re]{I}$ when $I$ is a
 finitely generated left ideal, a free $*-$Nullstellensatz.
However, this does not tell us for a particular ideal $I$ whether or  not 
 $I =\sqrt[\rm re]{I}$, 
 and that is the topic of this paper.
We give a complete solution for 
monomial ideals and  homogeneous principal ideals.
We also present the case of principal univariate ideals
with a degree two generator and find that it is very messy.
We  discuss an algorithm to determine if $I=\sqrt[\rm re]{I}$ 
 (implemented  under \href{http://math.ucsd.edu/~ncalg}{\tt NCAlgebra})
 with finite run times and provable effectiveness.
\end{abstract}

\maketitle

\section{Introduction}
 
 The introduction begins with definitions and a little motivation for them.
 Then it sketches the main results of this paper together with links to where
 they are found.

\subsection{Zero Sets in Free Algebras}
 \label{subsec:zeros}
Let $\axs$ be the monoid freely
generated by $\x=(x_1,\ldots, x_g)$ and 
$\xs=(x_1^*,\ldots,x_g^*)$, i.e., $\axs$ consists of  \df{words} in the $2g$
noncommuting
letters $x_{1},\ldots,x_{g},x_1^*,\ldots,x_g^*$
(including the empty word $\emptyset$ which plays the role of the identity $1$).
Let $\FA$ denote the 
$\R$-algebra freely generated by $\x,\xs$, i.e., the elements of $\FA$
are \df{polynomials} in the noncommuting variables $\x,\xs$ with coefficients
in $\R$.  Equivalently, $\FA$ is the \df{free $\ast$-algebra} on $\x$.
The
length of the longest word in a noncommutative polynomial $f\in \FA$ is the
\df{degree} of $f$ and is denoted by $\deg( f)$. The set
of all words of degree at most $k$ is $\axs_k$, and $\FA_k$ is the vector
space of all noncommutative polynomials of degree at most $k$.

 Given a $g$-tuple $X=(X_1,\dots,X_g)$ of same size 
 square matrices over $\R$, write
 $p(X)$ for the natural evaluation of $p$ at
 $X$.  For  $S \subset \FA$ we introduce
 \[
 V(S)^{(n)} = \left\{ (X,v) \in (\R^{n\times n})^g\times \R^n \mid p(X)v=0 \text{ for every }   p \in S\right\},
 \]
 and define the \df{zero set} of $S$ to be
\[
V(S)= \bigcup_{n\in \N} V(S)^{(n)} = \{ (X,v) \mid  p(X)v=0 \text{ for every }   p \in S \} .
\]
To each subset $T$ of $\bigcup_{n\in \N} \big((\R^{n\times n})^g\times \R^n\big)$ we 
associate the \df{left ideal}
\[
\J(T) = \{p \in \FA \mid  p(X)v=0 \text{ for every } (X,v) \in T \}.
\]
  For a left ideal $I$ of $\FA$, we call
\[
\sq{I} := \J\big(V(I)\big) %=\bigcap_{(\pi,v) \in \pt_{\cC}(\chrisA) \atop I \subseteq \J(\{(\pi,v)\})} \J(\{(\pi,v)\})
\]
the \df{\rad}{} of $I$.\footnote{In \cite{CHMN} this radical was denoted $\sqrt[\Pi]{\color{white}a}.$  Since in this article only radicals
with respect to finite dimensional representations are considered, the $\Pi$ has been dropped.}
 Evidently $\sq{I}$ is a left ideal. 
We say that $I$ is 
\df{has the \lnss}
if $\sq{I}=I$.
Now we describe a class of ideals which has this property.

 A polynomial $p $ is  \df{analytic}
if it  has no transpose variables, 
that is,  no $x_i^*$.
For example,
 $p(x) = 1 + x_1x_2 + x_1^3 + x_2^5$ is analytic while
 $p(x) = 1 + x_1^*x_2 + (x_1^*)^3 + x_2^5$ is not analytic.
There is a strong Nullstellensatz for left ideals
generated by analytic polynomials in \cite{HMP}.
This strengthens an earlier result proved by Bergman \cite{HM};
see also \cite{BK} for a survey on noncommutative Nullstellens\"atze.

\begin{theorem}  
\label{thm:statz}
Let $p_1, \ldots, p_m $ be analytic polynomials, and
let $I $ be the left ideal generated by those $p_i$.
Then $\sq{I}=I$, i.e.,
$$
\big( \forall j \  p_j(X)v=0  \big) \implies  \ q(X) v =0 \ \  \qquad  
\text{iff}   \qquad \ \ 
 q = { f_1} p_1 + \cdots + { f_m } p_m .$$
\end{theorem}

We pause to make two remarks related to Theorem \ref{thm:statz}. 
First, no powers  $q^k$ are needed, contrary to the case
in the classical commutative Hilbert Nullstellensatz  or 
the real Nullstellensatz. This absence of powers
 seems to be the pattern for free algebra
situations.  

Secondly, while there are other notions of zero of a free polynomial, the
 one used here is particularly suited for studying left ideals and
 has proved fruitful in a variety of other contexts; e.g.~\cite{HMP,HM3}.
 One alternative notion is to say that $X$ is a zero of $p$ if $p(X)=0$.
 However, in this case, for $R\subset \bigcup_{n\in \N}(\R^{n\times n})^g,$  the set $\{p\mid p(X)=0 \mbox{ for all } X\in R\}$ is
 a two-sided ideal.  Another choice is to declare $X$ a zero of $p$ if 
 $p(X)$ fails to be invertible, but then $\{p\mid \det(p(X))=0 \mbox{ for all } X\in R\}$
 is not closed under sums.

What about ideals generated by  $p_j$ which  are not analytic?
To shed light on the basic question of
which ideals have the \lnss, we seek
an algebraic description of the \rad{} $\sqrt[]{I}$
similar to the notion of real radical in the classical
real algebraic geometry,
cf.~\cite[Chapter 4]{BCR},
\cite[Chapter 2]{Mar},  \cite[Chapter 4]{PD} or \cite{Sce}.
For this we introduced real ideals in \cite{CHMN}. Now we recall these definitions.

A left ideal $I$ of $\FA$ is said to be \df{real}
if for every $a_1,\ldots,a_r\in\FA$
such that
\[
\sum_{i=1}^r a_i^\ast a_i \in I+I^\ast ,
\]
we have that $a_1,\ldots,a_r \in I$.
Here  $I^*$ is the right ideal $
I^* = \{ a^* \mid a\in I \}.$
An intersection of a family of real ideals is a real ideal.
For a left ideal $J$ of $\FA$ we call the ideal
\[
\rr{J} = \bigcap_{\substack {I \supseteq J\\ I \text{ real}}} I
=\text{the smallest real ideal containing } J
\]
\df{the {\qrr}} of $J$.  
It is not hard to show for any left ideal $I$ that
$$ I \subset \rr{I} \subset \sq{I},$$
see \cite{CHMN}. 
The main result of \cite{CHMN} is a real Nullstellensatz
which states:

\bet[\protect{\cite[Theorem 1.6]{CHMN}}]
\label{thm:chmnMain}
A finitely generated left ideal $I$ in $\FA$
  satisfies  $ \rr{I} = \sq{I}$. Thus $I$ has the 
\lnss \ 
if and only if it is real.
\eet
\noindent
This result is not true for infinitely generated ideals,
as is shown in Example \ref{ex:nonss}.

A quantitative version of this theorem 
gives bounds (which we shall need) on the degrees 
of the polynomials involved.

\bet [\protect{\cite[Theorem 2.5]{CHMN}}]
\label{thm:chmnDeg}
Let $I$ be a left ideal in $\FA$ generated by polynomials of
degree bounded by $d$.
Then $I$ is real if and only if
whenever $q_1, \ldots, q_k$ are polynomials with $
deg(q_j) < d$ for each $j$, and
 $\sum_{i=1}^{\ell}q_i^*q_i \in I + I^*,$ then $q_j\in I$ for each
 $j$.
\eet

These results give a clean equivalence but do not tell us 
whether or not
a particular ideal has the \lnss.  
This paper focuses on examples
of ideals $I$ for which we can determine if  
$I= \sq{I}$.

\subsection{Main Results}

Our goal is to determine  which left ideals
$I$ have the property $I= \sqrt[]{I}$. 
We give  a complete solution for principle ideals   
generated by a degree 1 or degree 2 polynomial. 
Other results treat more general situations but are less complete;
we list them below.
 
\subsubsection*{Monomial ideals}
A \df{left monomial ideal} is a left ideal  generated by monomials.  
A word $w\in\axs$ is \df{left unshrinkable} \cite{Lan,Tap} if it cannot be
written as $w=uu^*v$ for some $u,v\in\axs$ with $u\neq 1$.
We show:

\begin{thm}
\label{thm:shrinkMe}
A left monomial ideal is real if and only if
it is generated by left unshrinkable words.
Hence, by Theorem {\rm\ref{thm:chmnMain}}, a finitely generated left monomial ideal is real 
if and only if it has the \lnss.
\end{thm}

\noindent
We emphasize that first claim is proved in Section \ref{sec:mono} and does not
require finite generation.
Absent finitely many generators we need something more to
prove a version of the second claim,
see Example \ref{ex:nonss}.

\def\bs{\bigskip}

\subsubsection*{Principal ideals}
\begin{theorem}
\label{thm:principalI}
 Suppose $I \subset \FA$ is the left ideal generated by 
 a nonzero   polynomial $p$. 
 \ben[\rm (1)] 
\item
\label{it:psqf}
Suppose $p$  is homogeneous.
Then $I$ is real if and only if
$p$ is not of the form
\beq\label{eq:bad}
p = (s + q)f,
\eeq 
where
$s$ is a nonconstant sum of squares,
$q^* = -q$, and $f\in\FA$.
\item
\label{it:psqfih}
Even for nonhomogeneous $p$ the condition in \eqref{eq:bad}
 on $p$
rules out $I$ being real.
 \item
 \label{it:pluriharm}
Suppose $p = a +
b^*$, where $a$ and $b$ are analytic polynomials.
Then $I$ is not real if and only if  $p = a - a^* + c$ 
$($or equivalently, $p = -b +b^* + c)$
for some nonzero constant $c$.
\een
\end{theorem}

\noindent
Proving this theorem is the subject of Section \ref{sec:hom}.

\subsubsection*{Analytic homogeneous ideals}
 Turning to the question of  identifying not necessarily finitely generated
 left ideals with the \lnss, it turns out that an ideal generated by
 analytic homogeneous polynomials has the \lnss. The formal result
 is stated as  Proposition \ref{prop:infana}.

\subsubsection*{Algorithms}
We give a computer algorithm to determine if
a given $I$ is a real ideal, hence if $I$ has the \lnss, see
Subsection \ref{sec:poor}.
This is a more practical offshoot of an algorithm given 
in \cite{CHMN}.

Our classification of univariate real  ideals generated by  $p$  of degree 2
is done in Section  \ref{sec:quad}  and it 
illustrates techniques derived in Section \ref{sec:hom}.

\section{Monomial Ideals}\label{sec:mono}

In this section we describe what is known about monomial ideals
both finitely and infinitely generated. 
We start with basic facts about monomial ideals.

\begin{lem}\label{lem:monIdeal}
For a left ideal $ I\subseteq\FA$ the following are
equivalent:
\ben[\rm (i)]
\item
$ I$ is a left monomial ideal;
\item
A polynomial $f\in \FA$ is in $I$ iff all of the monomials
appearing in $f$ are in $I$.
\item
A polynomial $f\in \FA$ is in $I + I^*$ iff all of the monomials
appearing in $f$ are in $I + I^*$.
\item
A polynomial $f\in \FA$ is in $I + I^*$ iff all of the monomials
appearing in $f$ are in $I$ or $I^*$.
\een
\end{lem}

\begin{proof}
(i) $\Rightarrow$ (ii) is obvious. If (ii) holds, then $ I$ is generated by the set of
all monomials in $ I$ so (i) holds. The equivalence of (i) and (iii) follows similarly.
It is clear that (iii) and (iv) are equivalent.
\end{proof}

\subsection{Proof of Theorem \ref{thm:shrinkMe}}

Suppose $ I$ is generated by 
left unshrinkable words.
Our goal is to show that $I$ is a real ideal. (We emphasize that this part of the proof does not require finite generation.)

Assume $I$ is not real and consider a finite sum of hermitian squares
\[
s=\sum_{j} g_j^* g_j
\]
such that $s \in I+I^*$ and $g_j \not\in I$ for some $j$. Decompose each $g_j$ as
\[g_j = \sum_{w\in\axs} a_{w,j} w =
 \sum_{\substack{w\in\axs\\w \in  I}} a_{w,j}w + \sum_{\substack{u\in\axs\\u \not\in  I}} a_{u,j}u= p_j+q_j\]
 for some $a_{v,k}\in \R$. 
 Here $p_j= \sum_{w \in  I} a_{w,j}w \in I$ and $q_j=\sum_{u \not\in  I} a_{u,j}u$
 is either zero (if all of its coefficients are zero) or does not belong to $I$ (it it has a nonzero coefficient).
 Since
\[
s = \sum_{j} g_j^*g_j = \sum_j (p_j^*+q_j^*)(p_j+q_j) = \sum_j \left(p_j^*p_j+ p_j^*q_j+q_j^*p_j+
q_j^*q_j\right)
\]
belongs to $I +  I^*$, it follows that 
\begin{equation}
 \label{eq:sosNotInI}
\sum_j q_j^*q_j \in  I +  I^*.
\end{equation}
Moreover, since $g_j \not\in I$ for some $j$, also $q_j \ne 0$ for some $j$. We may assume that all $q_j$
in \eqref{eq:sosNotInI} are nonzero and hence each nonzero term of each $q_j$ in \eqref{eq:sosNotInI} lies outside $I$.
Since the highest degree terms in a sum of squares cannot cancel,
the highest degree terms in $\sum_j q_j^*q_j$ are multiples of words of the form $w^*w$ for words $w \not\in  I$.
Lemma \ref{lem:monIdeal} implies that $w^*w \in I + I^*$ for each such $w$.
Since $I$ is a left monomial ideal, it follows that $w^*w$ is of the form $v_1v_2$, where $v_2$ is one of the unshrinkable words which generate $I$.

If $\deg(w) \geq \deg(v_2)$, then this implies that $w = v'v_2$, where $v'$ is the right-hand piece of $v_1$,
which is a contradiction since $w \not\in I$. 

If $\deg(w) < \deg(v_2)$, then
$v_2 = w'w$, where $w'$ is the left-hand piece of $v_2$, and $\deg(w') > 0$.
Thus $v_1v_2 = v_1w'w = w^*w$, which implies that $w^* = v_1w'$,
which implies that $w = (w')^*v_1^*$.  Therefore $v_2 = w'(w')^*v_1^*$,
which is a contradiction since $v_2$ is unshrinkable.

\smallskip
Conversely, suppose $I$ is a real left monomial ideal.
Let $\mathcal B$ be a minimal set of words which generate $I$.
Without loss of generality, we may assume that for each $w \in \mathcal B$,
there exist no $u,v\in\FA$, with $v \in I$, such that $w = uv$.
Assume $w \in \mathcal B$ is shrinkable.  Let
$w = u^*uv$, with $u$ some nonconstant word.
Thus $v^*u^*uv \in I + I^*$, which implies that
$uv \in I$ by $I$ being real.  This however contradicts
the minimality of $\mathcal B$.

\smallskip
The second part of the theorem now follows from 
Theorem \ref{thm:chmnMain}, and this does require finite
generation; cf.~Example \ref{ex:nonss}.
\qed

\subsection{Real monomial ideals without the \lnss}
\label{sec:nonfin}

When does a left monomial ideal have the \lnss?
Theorem \ref{thm:shrinkMe} gives us a necessary condition,
namely that it is generated by left unshrinkable words.
However, this condition is not sufficient, see Example \ref{ex:nonss}.
In the positive direction, some infinitely unshrinkably generated left ideals
have the \lnss, see Section \ref{sec:nonfin2}.

We begin with the following basic fact:
 if $I$ is the intersection of some family of finitely generated real left ideals $I_\alpha$, then
\beq
\label{eq:finiteIntersec}
\sq{I} \subset \bigcap_{\alpha} \sq{I_{\alpha}} = \bigcap_{\alpha} I_{\alpha}
= I \subset \sq{I}.
\eeq
Therefore, to show that a left ideal $I$ satisfies $I = \sq{I}$, it suffices
to show it is an intersection of finitely generated real left ideals.

 Now we give an example of a real monomial ideal 
 on which the real Nullstellensatz fails.
\begin{example}
\label{ex:nonss}
If $x=(x_1,x_2)$ and $I\subseteq\FA$ is the left ideal generated by the set
\[\big\{ x_1(x_2^*x_2)^dx_1 \mid d \geq 2\big\}\]
then $I$ is real but $\sq{I} \neq I$.
\end{example}

\begin{proof}
If
\[u^*u v = x_1(x_2^*x_2)^dx_1\]
for some words $u,v$ and for some $d$, then either $\deg(u) = 0$ or $u^*$ starts with $x_1$.
The latter case cannot happen since this would imply that $u$ ends with $x_1^*$ and there is no $x_1^*$ in $x_1(x_2^*x_2)^dx_1$.
Therefore for each $d$, the monomial $x_1(x_2^*x_2)^dx_1$ is left unshrinkable, which 
by Theorem \ref{thm:shrinkMe}
implies that the left ideal $I$ is real.

Next, by definition
\[
 \sq{I} = \bigcap_{(X,v) \in V(I)} \cI\big(\{(X,v)\}\big),
\]
so we will show that there is a polynomial $p \not\in I$ contained in each left
ideal $\cI\big(\{(X,v)\}\big)$ with $(X,v) \in V(I)$.
Suppose $(X,v) \in V(I)$, where $X = (X_1, X_2)$ is a pair of $n\times n$
matrices and $v\in\R^n$. The matrix ${X_2}^*X_2$ is positive
semidefinite, so we can decompose it as
\[ {X_2}^*X_2 = U^* \Lambda U,\]
where $U$ is unitary and $\Lambda$ is a diagonal matrix with nonnegative eigenvalues $\lambda_1, \ldots, \lambda_n$.
Let $w = UX_1v$ be a vector with $i^{\rm th}$ entry denoted $w_i$,
 and let $Y = X_1U^*$ be a matrix with $(i,j)$ entry  $y_{ij}$.  Then for each $d \geq 2$, we see
 \[
\begin{aligned}
\notag
X_1(X_2^*X_2)^d X_1v = Y\Lambda^d w 
&= \begin{bmatrix}
y_{11}& y_{12}& \ldots&y_{1n}\\
y_{21}&y_{22}&\ldots &y_{2n}\\
\vdots&\vdots&\ddots&\vdots\\
y_{n1}&y_{n2}&\ldots&y_{nn}
\end{bmatrix}
\begin{bmatrix}
\lambda_1^d& 0& \ldots&0\\
0&\lambda_2^d&\ldots &0\\
\vdots&\vdots&\ddots&\vdots\\
0&0&\ldots&\lambda_n^d
\end{bmatrix}
\begin{bmatrix}
w_1\\
w_2\\
\vdots\\
w_n
\end{bmatrix}\\
\notag
&=
\begin{bmatrix}
\sum_{j=1}^{n} y_{1j}\lambda_j^d w_j\\
\sum_{j=1}^{n} y_{2j}\lambda_j^d w_j\\
\vdots\\
\sum_{j=1}^{n} y_{nj}\lambda_j^d w_j
\end{bmatrix}
=
\begin{bmatrix}
0\\0\\ \vdots \\ 0
\end{bmatrix}.
\end{aligned}
\]

If  $\Lambda = 0$, then $X_1X_2^*X_2X_1v = 0$.
Otherwise, let $\xi_1, \ldots \xi_k$ be the distinct nonzero values of
$\lambda_1, \ldots, \lambda_n$. For each $d \geq 2$ and each $i$,
\[
\sum_{j=1}^k \left( \sum_{\lambda_{\ell} = \xi_j} y_{i\ell}w_{\ell} \right)
\xi_j^d = 0.
\]
Since the matrix
\[
\begin{bmatrix}
\xi_1^2 & \xi_2^2 & \ldots & \xi_k^2\\
\xi_1^3 & \xi_2^3 & \ldots & \xi_k^3\\
\vdots&\vdots&\ddots&\vdots\\
\xi_1^{k+1} & \xi_2^{k+1} & \ldots & \xi_k^{k+1}
\end{bmatrix}
=
\begin{bmatrix}
1& 1 & \ldots & 1\\
\xi_1 & \xi_2 & \ldots & \xi_k\\
\vdots&\vdots&\ddots&\vdots\\
\xi_1^{k-1} & \xi_2^{k-1} & \ldots & \xi_k^{k-1}
\end{bmatrix}
\begin{bmatrix}
\xi_1^2 &0 & \ldots & 0\\
0 & \xi_2^2 & \ldots & 0 \\
\vdots&\vdots&\ddots&\vdots\\
0 & 0 & \ldots & \xi_k^{2}
\end{bmatrix}
\]
has nonzero determinant, each
$\sum_{\lambda_i = \xi_j} y_{ij}w_j = 0.$
Therefore
\begin{align}
\notag
X_1^*X_2^*X_2X_1v = Y\Lambda w &=\begin{bmatrix}
\sum_{j=1}^{n} y_{1j}\lambda_j w_j\\
\sum_{j=1}^{n} y_{2j}\lambda_j w_j\\
\vdots\\
\sum_{j=1}^{n} y_{nj}\lambda_j w_j
\end{bmatrix}=\begin{bmatrix}
\sum_{j=1}^k \left( \sum_{\lambda_{\ell} = \xi_j} y_{1\ell}w_{\ell} \right)
\xi_j\\
\sum_{j=1}^k \left( \sum_{\lambda_{\ell} = \xi_j} y_{2\ell}w_{\ell} \right)
\xi_j\\
\vdots\\
\sum_{j=1}^k \left( \sum_{\lambda_{\ell} = \xi_j} y_{k\ell}w_{\ell} \right)
\xi_j
\end{bmatrix} = \begin{bmatrix}
0\\0\\ \vdots \\ 0
\end{bmatrix}.
\end{align}
Hence $x_1x_2^*x_2x_1 \in \sq{I}$.  However, $x_1x_2^*x_2x_1 \not\in I$
since it is not a multiple of any of the monomial generators of $I$, because it
has degree $4$ and the homogeneous generators of $I$ each have degree of at least $6$.
\end{proof}
Note that what the above implies is that if $I \subset J$, 
where $J$ is a finitely generated real left ideal, 
then $x_1x_2^*x_2x_1 \in J \setminus I$. 
 Therefore, the intersection of all real finitely 
 generated ideals containing $I$ necessarily contains
  $x_1x_2^*x_2x_1$.

\subsection{Several non-finitely generated ideals with the \lnss}
\label{sec:nonfin2}

In this section we study non-finitely generated left ideals that satisfy the \lnss.

A bright spot in our understanding are  left ideals with analytic
monomial generators. Analytic monomials are certainly unshrinkable
and as we shall see always generate real ideals.
In fact, this works more generally than for just monomials:

\begin{prop}
\label{prop:infana}
 Suppose $I$ is generated by homogeneous analytic polynomials $p_1, \ldots,
p_k, \ldots$.
Then $\sq{I} = I$.
\end{prop}

\begin{proof}
 For each $d\in\N$, let $I^{(d)}$ be the ideal generated by all
polynomials $p_i$ of degree  $\leq d$ as well as all analytic words of
degree $d+1$.  In this case, each $p_j \in I^{(d)}$, and each $I^{(d)}$ is
finitely generated.  Further, since $I^{(d)}$ is generated by analytic
polynomials, it is real.  Therefore
\[ \sq{I} \subset \bigcap_{d=1}^{\infty} I^{(d)}.\]
At the same time
\[ I = \bigcap_{d=1}^{\infty} I^{(d)},\]
since the elements of $I^{(e)}$ of degree bounded by $d$, where $e \geq d$,
are precisely the elements of $I$ of degree bounded by $d$, cf.~Lemma \ref{lem:monIdeal}.
\end{proof}

This compares to \cite{HMP} which required finite-generation by analytic polynomials
(which may not be homogeneous).

In distinction to the analytic case the 
 next monomial ideal $I$  we treat has generators with adjoints.
The generators are very simple  left unshrinkable monomials
but still it is a bit tricky to prove the ideal is real.
Indeed
 we shall write $I$  
 as the intersection of finitely-generated ideals which
are not monomial ideals and which are 
 not even generated by homogeneous polynomials.

\begin{example}
\label{ex:infgenGood}
Let $x=(x_1,x_2)$ and let $I\subset \FA$ be the left ideal generated by 
the set \[\{x_1(x_2^*x_2)^d \mid d\ge 0\}.\]
Then $\sq{I} = I$.
\end{example}

\begin{proof}
 Let $I_{\lambda}$ be the left ideal generated by
$x_1$ and $x_2^*x_2 - \lambda$, where $\lambda > 0$.
By the discussion at the beginning of section \ref{sec:nonfin},
the claim that $\sq{I} = I$ follows from Steps 1 and 2 below.

{\sc Step {\rm 1}}: $I_{\lambda}$ is real. \\
Since $I_{\lambda}$ is generated by polynomials
of degree bounded by $2$, we only have to check
that no sum of squares of polynomials outside of $I$ of
degree $1$ or less
is in $I_{\lambda} + I_{\lambda}^*$.
Since $x_1 \in I_{\lambda}$, we need only consider
squares of polynomials in the span of all monomials of degree $1$ or less
which are not equal to $I$.
Suppose that
\[ 
 p =  
\begin{bmatrix}
 x_1^*\\x_2\\x_2^*\\1
   \end{bmatrix}^*
A
\begin{bmatrix}
    x_1^*\\x_2\\x_2^*\\1
   \end{bmatrix} \in \left(I_{\lambda} + I_{\lambda}^*\right)
   \]
   is a symmetric polynomial, i.e., $A$ is symmetric.
   It is straightforward to show that $p$ must be in the span of $x_2^*x_2 - \lambda$ and $x_1 + x_1^*$.  Therefore,
     \[
  A = \left[
\begin{array}{cccc}
 0&0&0&\beta\\
0&\alpha&0&0\\
0&0&0&0\\
\beta&0&0&-\lambda \alpha
\end{array}
\right].
\]
If $p$ is a nonzero sum of squares, then $(0,0) \neq (\alpha, \beta) \in\R^2$,
which implies that
\[\beta = 0,\  \alpha \neq 0, \quad \text{and} \quad \alpha \left[
\begin{array}{cc}
1&0\\
0&-\lambda 
\end{array}
\right] \succeq 0.\]
But this contradicts
 $\lambda > 0$.

{\sc Step {\rm 2}}: $\bigcap_{\lambda > 0} I_{\lambda} = I$.\\
First, consider
\[x_1(x_2^*x_2)^d.\]
If $d = 0$, then $x_1 \in I_{\lambda}$.
Further, we see that
\[ x_1(x_2^*x_2)^d = \lambda x_1(x_2^*x_2)^{d-1} +
x_1(x_2^*x_2)^{d-1}(x_2^*x_2 - \lambda),\]
hence by induction $x_1(x_2^*x_2)^d \in I_{\lambda}$ for all $d$.
Since $I_{\lambda}$ contains the generators of $I$,
\[ I \subset \bigcap_{\lambda > 0} I_{\lambda}. \]

Next, suppose $p \in \bigcap_{\lambda > 0} I_{\lambda}$.
First, $p \in I_1$, so $p$ is of the form
\[ p = q x_1 + \sum_{i=0}^d r_i (x_2^*x_2 - 1)^i,\]
for some $d$, where $r_i, q \in \FA$.
Without loss of generality, assume that the terms of the $r_i$ are not divisible on the right
by $x_2^*x_2$; for example, if $r_i = \tilde{r}_ix_2^*x_2$
for some $\tilde{r}_i$, then
\[ r_i(x_2^*x_2 - 1)^i = \tilde{r}_i (x_2^*x_2)(x_2^*x_2 - 1)^i
= \tilde{r}_i(x_2^*x_2 - 1)^{i+1} + \tilde{r}_i (x_2^*x_2 - 1)^i . \]
Given another $\lambda > 0$, $p \in I_{\lambda}$.
We see that
\[ x_2^*x_2 - 1 + I_{\lambda} = \lambda - 1 +I_{\lambda},\]
and inductively that
\[(x_2^*x_2 - 1)^i + I_{\lambda} =  (\lambda - 1)^i + I_{\lambda}. \]
Therefore
\[\sum_{i=0}^d r_i (\lambda - 1)^i \in I_{\lambda}. \]
The leading term of this resulting polynomial
must be divisible on the right by $x_1$ or $x_2^*x_2$
since
\[ \{ x_1, x_2^*x_2 - \lambda \}\]
is a left Gr\"obner basis.
Since the
terms of each $r_i$ are not divisible
on the right by $x_2^*x_2$ by construction, the leading term of
$\sum_{i=0}^d r_i (\lambda - 1)^i$ is divisible on the
right by $x_1$.  Since this is true for arbitrary
$\lambda > 0$,  the leading term of
some $r_i$ is divisible on the right by $x_1$.
Let this term be denoted $r_i'$.  Then
\[ p - r_i'(x_2^*x_2)^i \in \bigcap_{\lambda > 0} I_{\lambda}\]
since $r_i'(x_2^*x_2)^i \in I \subset \bigcap_{\lambda > 0} I_{\lambda}$
and is in $I$ if and only if $p$ is.
We can reduce $p$ inductively to deduce that $p \in I$.
\end{proof}

\section{Principal Left Ideals}
\label{sec:hom}

We now turn our attention to principal (i.e., singly generated) left ideals.
Our focus here is on homogeneous generators. The main results
in this section are Theorem \ref{thm:prinMain} that precisely describes
when a principal homogeneous left ideal is real, and the accompanying
Algorithm \ref{algor:factor}, producing a practical test for determining whether a 
given homogeneous polynomial $p\in\FA$ generates a real left ideal.

\subsection{Factorization of Homogeneous Polynomials}

Checking whether a polynomial $p\in\FA$ factors as a product 
$p=p_1p_2$ of polynomials of smaller degree amount to solving a large
system of linear equations. For 
homogeneous noncommutative polynomials 
factorization can be done more efficiently, and this is what we describe
in this subsection.

\begin{lemma}
\label{lemma:goodtechnical}
Let $V(x) = (v_1, \ldots, v_r)^*\in \FA^r$
be a vector of linearly independent
homogeneous degree $d_r$  polynomials, and let  $W(x) = (w_1, \ldots, w_s)^*\in \FA^s$ 
be a vector of linearly independent
homogeneous degree $d_s$ polynomials. For
 $A \in \R^{r \times s}$, we have
\[V(x)^* A W(x) = 0
\quad\iff\quad
A = 0.
\]
\end{lemma}

\begin{proof}First consider the case where $V = M_{d_r}$ is a vector whose entries are all the words of degree $d_r$,
and where $W = M_{d_s}$ is a vector whose entries are all the words of degree $d_s$.  
Let $A \in \R^{r \times s}$ satisfy
$M_r(x)^* A M_s(x) = 0$.  Each word $w$ of degree $d_r + d_s$ can be  expressed uniquely as a product
\[w =w_r w_s, \]
where $w_r$ is the word consisting of the first $d_r$ letters of $w$ and
$w_s$ is the word consisting of the last $d_s$ letters of $w$.
Let $A_{w_r^*,w_s}$ be the entry of $A$ corresponding to $w_r^*$ on the left and $w_s$ on the right.
Then,
\[M_r(x)^* A M_s(x) = 
\sum_{\deg(w_r)= d_r} \sum_{\deg(w_s) = d_s} A_{w_r^*,w_s} w_rw_s = 0. \]
This implies that each $A_{w_r^*,w_s} = 0$, or in other words, that $A = 0$.

Next consider the case where the entries of $V$ span the set of all homogeneous degree $d_r$ polynomials, and where the entries of $W$ span the set of all homogeneous degree $d_s$ polynomials.  The elements of $V$ may be expressed uniquely as a linear combination of the elements of $M_r$.  Therefore there exists an invertible matrix $R$ such that
\[RV(x) = M_r(x). \]
Similarly, there exists an invertible matrix $S$ such that
\[SW(x) = M_s(x). \]
Therefore
\[ V(x)^* A W(x) = M_r^* (R^{-1})^* A S^{-1} M_s = 0 \]
which implies that $(R^{-1})^*AS^{-1} = 0$. Since $R$ and $S$ are invertible, it follows that $A = 0$.

Finally, consider the general case.  Let $V'$ be a vector whose entries, together with the entries of $V$, form a basis for the set of homogeneous degree $d_r$ polynomials, and let $W'$ be a vector whose entries, together with the entries of $W$, form a basis for the set of homogeneous degree $d_s$ polynomials. Suppose $V(x)^TAW = 0$.  Then
\[V^T A W = \begin{bmatrix} V\\V' \end{bmatrix}^T \left[ \begin{array}{cc}
A&0\\
0&0
\end{array}
\right]
\begin{bmatrix} W\\W' \end{bmatrix} =0,\]
which implies that $A = 0$ by previously established case.
\end{proof}

\begin{remark}
Let $M_{d_1} = (w_1^{(1)}, \ldots, w_k^{(1)})$ be a vector whose entries are
all words of length $d_1$ and
$M_{d_2}= (w_1^{(2)}, \ldots, w_{\ell}^{(2)})$ be a vector whose entries are
all words of length $d_2$.
Let $p$ be homogeneous of degree $d_1 + d_2$.
Then the unique matrix $A$ such that
\[ p = M_{d_1}^* A M_{d_2}\]
is the matrix whose $(i,j)$-entry is the coefficient
of $(w_i^{(1)})^*w_j^{(2)}$ in $p$. We call it the $(d_1,d_2)$-\df{Gram matrix} for $p$.
$($Observe that the uniqueness of the Gram matrix fails for nonhomogeneous polynomials.$)$
\end{remark}

The following is obvious and well-known.

\begin{lem}[cf.~\cite{Har+,KP}]\label{lem:homsos}
 Let $p \in \FA$ be a homogeneous degree $2d$ polynomial.
Then $p$ is a sum of squares if and only if its $(d,d)$-Gram matrix 
is positive semidefinite.
\end{lem}

\begin{comment}
\begin{proof}
 If $p = M_{d}^*AM_d$, and $A \succeq 0$, then
\[ p = \left(\sq{A} M_d \right)^*\left(\sq{A} M_d \right),\]
which implies that $p$ is a sum of squares.

Conversely, suppose $p = p_1^*p_1 + \cdots + p_k^*p_k$ is a sum of squares.
Let $p_i'$ be the terms of $p_i$ of degree $\max_j\{\deg(p_j)\}$.
Then
\begin{equation}
\label{eq:homogSOS}
  p = \sum_{i=1}^k
  \Big( (p_i')^*(p_i') + (p_i- p_i')^*(p_i') + (p_i')^*(p_i-
p_i') + (p_i- p_i')^*(p_i- p_i') \Big). 
\end{equation}
The first term of \eqref{eq:homogSOS} yields a sum of squares of degree
$2\max_i\{\deg(p_i)\}$ and the other terms have degree $<2\max_i\{\deg(p_i)\}$.
Since $p$ is homogeneous, and since a sum of nonzero squares is nonzero,  each of the $p_i$ is homogeneous of degree $d$.
  For each $i$, let $p_i = \alpha_i^*M_d$ for
some scalar vector $\alpha_i$ of suitable dimension.
Then
\[ p = M_d^* \left( \sum_{i=1}^k \alpha_i\alpha_i^* \right) M_d,\]
and so $\sum_{i=1}^k \alpha_i\alpha_i^* \succeq 0$ is the $(d,d)$-Gram matrix
for $p$.
\end{proof}
\end{comment}

We call a polynomial $p\in\FA$ \df{irreducible} if it cannot 
be written as a product of two polynomials of smaller degree,
i.e., if $p=qr$ with $q,r\in\FA$, then $q$ is constant or $r$ is constant.

\begin{prop}[cf.~\protect{\cite{Co63}}]
\label{prop:factorize}
 Let $p \in \FNSx$ be a noncommutative homogeneous polynomial.
Then $p$ can be factored as
\[ p = p_1 \cdots p_k\]
for some irreducible homogeneous polynomials
$p_1, \ldots, p_k$.
Further, such a representation is unique in that
if $p = q_1 \cdots q_{\ell}$ is another
decomposition, where each $q_j$ is irreducible,
then $k= \ell$  and there exist nonzero scalars
$\lambda_1, \ldots, \lambda_k$ such that
$q_i = \lambda_i p_i$ for $i=1,\ldots,k$.
\end{prop}

We refer the reader to \cite{Co06} for a detailed study of factorization
in free algebras, and to
\cite{GGRW} for another take on noncommutative factorization.

\smallskip
It is straightforward to compute a
factorization of a homogeneous noncommutative polynomial.
There is an algorithm described in \cite{c}; see also \cite{KaS93}.
An alternate way of looking at the algorithm is presented
in Algorithm \ref{algor:factor} below. 

\begin{lem}\label{lem:alg}
A homogeneous polynomial $p\in\FA$
can be factored as $p = p_1p_2$, for
$\deg(p_1) = d_1$ and $\deg(p_2) = d_2$
if and only if the $(d_1,d_2)$-Gram matrix $A$ 
for $p$ has rank $1$.
Indeed, a decomposition $p = p_1p_2$ is then given by
decomposing $A = \alpha_1 \alpha_2^*$
and setting $p_1 = M_{d_1}^* \alpha_1$ and
$p_2 = \alpha_2^*M_{d_2}$.
\end{lem}

\begin{proof}
Straightforward.
\end{proof}

\begin{algor}
\label{algor:factor}
Suppose $p\in\FA$ is  a homogeneous degree $d$ polynomial.
If its $(d_1,d_2)$-Gram matrix is of rank $>1$ for all $d_1,d_2\in\N$ with
$d_1+d_2=d$, then $p$ is irreducible.
Otherwise choose the smallest possible $d_1$ producing a factorization
$p=p_1 \tilde p_1$ as in Lemma {\rm\ref{lem:alg}}. Then $p_1$ is irreducible.
Repeating the procedure on $\tilde p_1$ yields a factorization $\tilde p_1=p_2\cdots
p_k$, where each $p_j$ is irreducible. 
Then $p=p_1p_2\cdots p_k$.
\end{algor}

\subsection{Homogeneous Principal Left Ideals}

There is a clean test for determining whether or not a principal left ideal is real.
This test does not require homogeneity.

\begin{prop}
\label{prop:degBoundsPI}
Let $I$ be the left ideal generated by a $($not necessarily homogeneous$)$ polynomial $p\in\FA$. Then $I$ is real if and
only
if there exists no nonzero sum of squares equal to 
$qp + p^*q^*$ for a polynomial $q\in\FA$ with $\deg(q) < \deg(p)$.
\end{prop}

\begin{proof}
The left ideal $I$ is equal to
$$ I = \{ qp\mid q \in \R\axs\}.$$
A polynomial $qp$ has $\deg(qp) = \deg(q) + \deg(p)$, which implies that
$I$ contains no nonzero polynomials of degree less than $\deg(p)$.

By Theorem \ref{thm:chmnDeg},  $I$ is real if and only if there exists
no sum of squares of the form
\begin{equation}
\label{eq:sosOfForm}
 q_1^*q_1 + \cdots + q_k^*q_k \in I + I^*,
\end{equation}
with $\deg(q_j) < \deg(p)$ for each $j$, and $q_j \not\in I$.
Since $I$ is principal, as shown above there are no nonzero $q_j \in I$
with $\deg(q_j) < \deg(p)$.
Further, a sum of squares of the form (\ref{eq:sosOfForm}) has degree equal to
$2\max\{ \deg(q_j) \}$.
Therefore $I$ is real if and only if there exists no nonzero sum of squares
in $I + I^*$
with degree less than $2\deg(p)$.

\cite[Proposition 2.18]{CHMN} implies
that $(I + I^*)_{2d-1} = I_{2d-1} + I_{2d-1}^*$. The set $I_{2d-1}$ is equal to
\[ I_{2d-1} = \{ qp\mid \deg(q) < \deg(p)\}\]
since $\deg(qp) = \deg(q) + \deg(p)$.
Therefore an element of $(I + I^*)_{2d-1}$ is of the form
$q_1p + p^*q_2^*$, with $\deg(q_1), \deg(q_2) < \deg(p)$.  Further, if $q_1p +
p^*q_2^*$ is symmetric, then
$$ q_1p + p^*q_2^* = \frac{1}{2}(q_1 p + p^*q_2) + \frac{1}{2}(q_1p+p^*q_2)^*=
\left(\frac{1}{2}(q_1 + q_2) \right)p + p^* \left(\frac{1}{2}(q_1 + q_2)
\right)^*.$$
Therefore the set of symmetric elements of $(I + I^*)_{2d-1}$ is
$\{qp + p^*q^*\mid \deg(q) < \deg(p) \}$.
Hence $I$ is real if and only if there exists no nonzero sum of squares of
the form $qp + p^*q^*$ with $\deg(q) < \deg(p)$.
\end{proof}

\begin{example}
 Suppose $I \subset \FA$ is the left ideal generated by
a homogeneous polynomial $p$, and $p$ has no terms containing $x_ix_i^*$ or
$x_i^*x_i$ for any $i$.  Then $I$ is real.
\end{example}

\begin{proof}
By Proposition \ref{prop:degBoundsPI}, it suffices to assume that there exists a
polynomial $q$ such that $qp + p^*q^*$ is a nonzero sum of squares, with
$\deg(q) <
\deg(p)$.
Then $\deg(qp + p^*q^*) \leq \deg(q) + \deg(p) < 2\deg(p)$.
A nonzero sum of squares $s$ necessarily contains some terms of the form $w^*w$,
where
$2\deg(w) = \deg(s)$.  Such a term contains a $x_ix_i^*$ or a $x_i^*x_i$ at
the $\deg(s)/2$ through $\deg(s)/2 + 1$ position.  However, no term of $qp + p^*q^*$
contains  such a hermitian square  since $p$ contains no such terms and $2\deg(p) > \deg(qp
+ p^*q^*)$.
\end{proof}

\begin{lemma}
 \label{lem:factorSOS}
Let $p, q$ be homogeneous polynomials with
$\deg(q) \leq  \deg(p)$.  Then $qp + p^*q^*$ is a sum of squares
if and only if $p = rq^*$ for some polynomial $r$ such that
$r + r^*$ is a sum of squares.
\end{lemma}

\begin{proof}

Suppose $\deg(qp +  p^*q^*) = 2d $ so that $d \leq \deg(p)$.
Let $w_1q^*, \ldots, w_kq^*, r_1, \ldots, r_{\ell}$
be a basis for 
homogeneous polynomials of degree $d$
so that $w_1, \ldots, w_k$ are all
words of length $d - \deg(q)$.
Then there exists a unique decomposition of $qp$ as
\[qp =  \begin{bmatrix}
    w_1 q^*\\
\vdots\\
w_k q^*\\
r_1\\
\vdots\\
r_{\ell}
   \end{bmatrix}^*
\left[
\begin{array}{cc}
 A&B\\
0&0
\end{array}
\right]
\begin{bmatrix}
    w_1 q^*\\
\vdots\\
w_k q^*\\
r_1\\
\vdots\\
r_{\ell}
   \end{bmatrix}
\]
for some block matrices $A$ and $B$.
Further,
\[qp + p^*q^* = \begin{bmatrix}
    w_1 q^*\\
\vdots\\
w_k q^*\\
r_1\\
\vdots\\
r_{\ell}
   \end{bmatrix}^*
\left[
\begin{array}{cc}
 A+A^*&B\\
B^*&0
\end{array}
\right]
\begin{bmatrix}
    w_1 q^*\\
\vdots\\
w_k q^*\\
r_1\\
\vdots\\
r_{\ell}
   \end{bmatrix}. \]
Since $qp + p^*q^*$ is a sum of squares,
and by uniqueness of the Gram matrix representation,
we have $B = 0$.
Therefore
\[ qp =
\begin{bmatrix}
   w_1 q^*\\
\vdots\\
w_k q^*\\
r_1\\
\vdots\\
r_{\ell}
   \end{bmatrix}^*
\left[
\begin{array}{cc}
 A&0\\
0&0
\end{array}
\right]
\begin{bmatrix}
    w_1 q^*\\
\vdots\\
w_k q^*\\
r_1\\
\vdots\\
r_{\ell}
   \end{bmatrix} =
q \begin{bmatrix}
   w_1\\
\vdots\\
w_k
   \end{bmatrix}^*
A
\begin{bmatrix}
    w_1 \\
\vdots\\
w_k
   \end{bmatrix}q^*\]
Hence
\[ r =  \begin{bmatrix}
   w_1\\
\vdots\\
w_k
   \end{bmatrix}^*
A
\begin{bmatrix}
    w_1 \\
\vdots\\
w_k
   \end{bmatrix}\]
gives $p = rq^*$ and $r + r^*$ is a sum of squares.
The converse implication is trivial.
\end{proof}

\begin{cor}
\label{cor:testByDecomp}
 Let $I \subset \FNSx$ be the left ideal
generated by a single homogeneous
polynomial $p$.  Let $p = p_1 \cdots p_k$
be a factorization of $p$ into irreducible
factors.
Then $I$ is real if and only if none of the $2k$ polynomials
\[ \pm (p_1 + p_1^*), \ \pm (p_1p_2 + p_2^*p_1^*),
\  \ldots,  \ \pm (p_1 \cdots p_k + p_k^* \cdots p_1^*)\]
is a nonzero sum of squares.
\end{cor}

For the sake of convenience, we call the highest degree homogeneous
part of a polynomial $p\in\FA$ its \df{leading polynomial}. 

\begin{proof}
 By Proposition \ref{prop:degBoundsPI},
$I$ is real if and only if
 there is no nonzero sum of squares of the form
$qp + p^*q^*$,
where $\deg(q) < \deg(p)$.
We claim that we can restrict without loss of generality to $q$ being homogeneous,
i.e., $q = q'$.
Indeed,
if $q'$ is the leading polynomial of $q$, then
either $q'p + p^*(q')^* = 0$ or
$q'p + p^*(q')^*$ is the leading polynomial
of $qp + p^*q^*$.
In the former case,
\[ qp + p^*q^* = (q-q')p + p^*(q-q')^*,\]
so we can look at the leading polynomial of
$q-q'$ and repeat.
In the latter case,
$q'p + p^*(q')^*$ is a sum of squares since $qp + p^*q^*$ is.

If $s=(-1)^t(p_1p_2\cdots p_j + p_j^*\cdots p_2^*p_1^*)$ is a nonzero sum of squares for some
$j < k$ and some $t$, then let $q = (-1)^t p_k^*p_{k-1}^*\cdots p_{j+1}^*$. 
We have $\deg(q) < \deg(p)$ and
\[ q p + p^*q^* = p_k^*p_{k-1}^*\cdots p_{j+1}^* s p_{j+1}\cdots p_{k-1}p_k\]
is a nonzero sum of squares in $I+I^*$. If $s=(-1)^t(p+p^*)$ is a nonzero sum of squares
for some $t$, then $q=(-1)^t$ works.

Conversely, suppose $qp +  p^*q^* $  is homogeneous
of degree $2d$, for some $d < \deg(p)$, and assume it is a nonzero sum of
squares.  By Lemma \ref{lem:factorSOS}, $p = rq^*$ for some $r$ such that $r +
r^*$ is a sum of squares.  
By the uniqueness of the factorization,
\[ r = \lambda p_1 \cdots p_j \quad \text{and} \quad q^* =
\frac{1}{\lambda} p_{j+1} \cdots p_k, \]
for some scalar $\lambda$ and some $j$, which implies, since $r+r^*$
is a nonzero sum of squares, that
\[ (-1)^t( p_1 \cdots p_j + p_j^* \cdots p_1^*)\]
is a nonzero sum of squares for some $t$.
\end{proof}

Irreducibility of the factors $p_j$ allow a rephrasing of Corollary
\ref{cor:testByDecomp}
which lends itself to computation.  

\begin{thm}
\label{thm:prinMain}
 Let $p$ be a nonconstant homogeneous polynomial,
and let $I$ be the left ideal generated by $p$.
Decompose $p$ as $p = p_1 \cdots p_{k}$,
where the $p_i$ are irreducible and nonconstant.
Then $I$ is real if and only if neither
of the following two hold:
\begin{enumerate}[\rm (1)]
 \item
  \label{item:1} There exists $j$ with $2j \leq k$ such that
\[ p_1 = \lambda_1 p_{2j}^*, \ p_2 = \lambda_2 p_{2j - 1}^*,
\  \ldots, p_j = \lambda_j p_{j+1}^*\]
for some scalars $\lambda_i$.
\item
\label{item:2}
There exists $j$ with  $2j + 1 \leq k$ such that
$$
p_1 = \lambda_1 p_{2j + 1}^*, \ p_2 = \lambda_2 p_{2j}^*,
 \ \ldots, \  p_j = \lambda_j p_{j+2}^*
 $$
for some scalars $\lambda_i$,
and  $(-1)^t(p_{j+1} + \;  p_{j+1}^*)$   is a nonzero sum of squares for some $t$.
\end{enumerate}
\end{thm}

\begin{remark}
\label{rem:SOS}
 This theorem implies that
to test the conditions of Corollary
{\rm\ref{cor:testByDecomp}} for $p=p_1\cdots p_k$,
one only needs to do the following.
For each $j$ with $2j \leq k$,
examine whether for each $i$ the polynomials
$p_i$ and $p_{2j-i}^*$ are scalar multiples
of each other.
For each $j$ with $2j + 1 \leq k$,
examine whether for each $i$ the polynomials
$p_i$ and $p_{2j+1-i}^*$ are scalar multiples
of each other, and if they are, check to see if
$p_{j+1} + p_{j+1}^*$
is plus or minus a sum of squares.
Checking if a homogeneous degree $2d$ noncommutative polynomial
$p$
is a sum of squares is straightforward: by Lemma
{\rm\ref{lem:homsos}} we only need to
find its $(d,d)$-Gram matrix $($by solving a linear system$)$ and
test whether it is
 positive semidefinite.
\end{remark}

\begin{proof}[Proof of Theorem {\rm\ref{thm:prinMain}}]
 By Corollary \ref{cor:testByDecomp},
$I$ is not real if and only if
for some $\ell$ and $t$,
\[ (-1)^t(p_1 \cdots p_{\ell} + p_{\ell}^* \cdots p_1^*)\]
is a nonzero sum of squares.
We claim that the latter is true if and only if either 
item (1) or item (2) of Theorem \ref{thm:prinMain} holds.
We will prove the nontrivial direction of this claim by induction on $\ell$.

If $\ell = 1$, then $(-1)^t(p_1 + p_1^*)$ is a sum of squares for some $t$
and so item \eqref{item:2} holds.

If $\ell > 1$, suppose that $\deg(p_{\ell}) \leq \frac{1}{2} \deg(p_1 \cdots
p_\ell)$. If
this were not the case, a similar argument could be made using $p_1$ noting that
$\deg(p_{\ell}) > \frac{1}{2}\deg(p_1 \ldots p_{\ell})$ implies that $\deg(p_1)
\leq \frac{1}{2} \deg(p_1 \cdots p_{\ell})$.
Note that $(-1)^t(p_1 \cdots p_{\ell} + p_{\ell}^* \cdots p_1^*)$ is a sum of
homogeneous degree $\deg(p_1 \cdots p_{\ell})$ polynomials and is
a nonzero sum of squares.  Therefore
\[ \deg(p_1 \cdots p_{\ell} + p_{\ell}^* \cdots p_1^*) = \deg(p_1 \cdots
p_{\ell})\]
and this degree is even.
By Lemma \ref{lem:factorSOS}, this implies that $p_1^*$ factors $p_2 \cdots
p_{\ell}$ on the right.  Since $p_1$ is irreducible, by uniqueness of factorization
there exists some nonzero $\lambda_1$ such that $p_1 = \lambda_1 p_{\ell}^*$.
Now either ${\ell} = 2$ and so item \eqref{item:1} holds, or
$p_2 \cdots p_{{\ell}-1}$ satisfies
\[ (-1)^t(p_2 \cdots p_{{\ell}-1} + p_{{\ell}-1}^* \cdots p_2^*)\]
is a sum of squares.  
The result follows by induction.
\end{proof}

Now we prove Theorem \ref{thm:principalI} \eqref{it:psqf}.
\begin{cor}
\label{cor:specialForm}
 Let $0\neq p \in \FNSx$ be homogeneous
and let $I$ be the left ideal
generated by $p$.
Then $I$ is real if and only if
$p$ is not of the form
\beq\label{eq:badd}
p = (s + q)f,
\eeq 
where
$s$ is a nonconstant sum of squares,
$q^* = -q$, and $f\in\FA$.
\end{cor}

\begin{proof}
 Theorem \ref{thm:prinMain}
implies that $I$ is not real if and only if
one of two cases hold.
In case \eqref{item:1}, $p$ is equal to
$p = u^*uv$ for some nonconstant polynomial $u$. Here $s = u^*u$ and $q = 0$ gives the
result.
In case \eqref{item:2}, $p$ is equal to
$p = u^*tuv$, with $t + t^*$ plus or minus a sum of squares.
By multiplying $v$ by $\pm 1$, we can assume without loss of generality that $t + t^*$ is a sum of squares.
Setting $s = \frac 12 u^*(t + t^*)u$, $q = \frac12u^*(t - t^*)u$, and $f = v$
gives the result. 
The converse implication (if $p$ is of the form \eqref{eq:badd} then 
$I$ is non-real) is trivial;  see Proposition \ref{prop:1} below.
\end{proof}

\begin{example}
Let $w\in\FA$.
If $w = (s + q)r$, for some nonconstant sum of squares $s$, some
antisymmetric $q$, and some $r$, then $w$ is a monomial
if and only if $s+q$ and $r$ are monomials, that is, $s = u^*u$ for
some monomial $u \in \FA$, $q = 0$, and $r$ is a monomial.
Therefore, Corollary {\rm\ref{cor:specialForm}} in the monomial case states
that the left ideal $I$ generated by a monomial $w$ is real if and only if $w$ is
left unshrinkable.
\end{example}

\subsection{General Principal Left Ideals}
\label{sec:nhom}

Apart from Proposition \ref{prop:degBoundsPI} we have to this point considered only homogeneous ideals.
Now we consider principal left ideals with a  nonhomogeneous generator.
Here the results are less definitive than before.

Now we prove Theorem \ref{thm:principalI} \eqref{it:psqfih}.
If $p \in \FNSx$ is
not homogeneous but has a similar structure to
that of Corollary \ref{cor:specialForm},
then the following holds.

\begin{prop}\label{prop:1}
 If $p$ is of the form
\begin{equation}
 \label{eq:specialForm}
 p = (s + q)f,
\end{equation}
where $s$ is a nonzero sum of
squares, $q^* = -q$,
$f \neq 0$,
and $\deg(s + q) > 0$,
then the left ideal $I$ generated
by $p$ is not real.
\end{prop}

\begin{proof}
Since $\deg(s + q) > 0$, and $\deg(p) = \deg(f) + \deg(s+q)$, we have $\deg(f)
< \deg(p)$.
We see that
\[  f^*p + p^*f = 2f^*sf \in I + I^*\]
is a nonzero sum of squares.  By Proposition
\ref{prop:degBoundsPI}, this implies that $I$ is not real.
\end{proof}

\begin{example}
 Not every non-real principal left ideal
is generated by a polynomial of the form
\eqref{eq:specialForm}.  Consider
\[p = (x_1 - x_1^* + 1)(x_2 - x_2^*) + 1.\]
One can show $p$ is irreducible.
Further,
\[ p + p^* = (x_1 - x_1^*)(x_2 - x_2^*) + (x_2 - x_2^*)(x_1 - x_1^*) + 2\]
which is not a $(\pm)$ sum of squares.
However,
\[ (x_2 - x_2^*)p + p^*(x_2 - x_2^*)^* = 2(x_2 - x_2^*)^*(x_2-x_2^*)\]
is a sum of squares of elements not in the left ideal generated by $p$.
\end{example}

\begin{prop}\label{prop:2}
 If $p$ is of the form
\begin{equation}
\label{eq:conjForm}
  p = (s + q_1)q_2 + c,
\end{equation}
where each $q_i$ is antisymmetric,
$q_2\neq0$,
$s$ is a nonzero sum of squares,
$c$ is a constant,
and $\deg(s + q_1) > 0$
then the left ideal $I$ generated by $p$
is not real.
\end{prop}

\begin{proof}
Since $\deg(s + q_1) > 0$, we see $\deg(q_2) < \deg(p)$.
We see 
\[q_2^*p + p^*q_2 = 2q_2^*sq_2 \in I + I^*\]
is a nonzero sum of squares, which implies that
$I$ is not real.
\end{proof}

However, Propositions \ref{prop:1} and \ref{prop:2} do not describe all polynomials generating
 non-real principal left ideals.

\begin{ex}
Consider the following univariate polynomial
\[
p=xx^*-(x^*)^2+2x+4.
\]
As is easily seen, $p$ is not of the form \eqref{eq:specialForm} or \eqref{eq:conjForm}.
On the other hand, the left ideal $I$ generated by $p$ is non-real. Indeed,
\[
(x+2) p + p^*(x+2)^* = (4+2x^*)^* (4+2x^*) \in I+I^*,
\]
but $4+2x^*\not\in I$. We shall investigate reality of ideals generated by quadratics in more
detail in Section {\rm\ref{sec:quad}} below.
\end{ex}

\begin{prop}
\label{prop:leadingFactorSquares}
 Let $I$ be the left ideal generated by some polynomial $p$.
Let $p'$ be the leading polynomial of $p$, and let $p' = p_1 \cdots p_k$ be a
factorization of $p'$ into irreducible parts.
If $I$ is not real, then at least one of
\[ \pm(p_1 + p_1^*), \pm(p_1p_2 + p_2^*p_1^*), \ldots, \pm(p_1\cdots p_k + p_k^*
\cdots p_1^*)\]
is a sum of squares.
\end{prop}

\begin{proof}
 By Proposition \ref{prop:degBoundsPI}, if $I$ is not real, then there exists
a $q$ with $\deg(q) < \deg(p)$ and $qp + p^*q^*$ being a nonzero sum of squares.
Let $q'$ be the leading polynomial of $q$.
Then
\[ q'p' + (p')^*(q')^*\]
is either $0$ or the leading polynomial of $qp + p^*q^*$, in which case it is a
sum of squares.  Hence, in either case, $q'p' + (p')^*(q')^*$ is a sum of
squares.
By Lemma \ref{lem:factorSOS}, this necessarily
implies that $p' = p_1 \cdots p_j (\lambda q')$ for some $j$ and some nonzero
$\lambda$, and that
$\pm(p_1 \cdots p_j + p_j^* \cdots p_1^*)$ is a sum of squares.
\end{proof}

\subsection{Analytic plus antianalytic generators}
\label{subsec:further}
Recall that a polynomial $p\in\FA$ is  analytic if it contains no $x_j^*$, i.e.,
if $p\in\R\ax$.
Likewise, a polynomial $p\in\R\axstar$ is called {\bf antianalytic}.

\begin{lemma}
\label{lem:antiHomIrr}
 Let $a, b\in \FA$ be homogeneous analytic polynomials such that $\deg(a) =
\deg(b) > 0$.
Then $a + b^*$ is irreducible.
\end{lemma}

\begin{proof}
 Assume that $a+b^*$ is reducible. Proposition \ref{prop:factorize} then implies that $a+b^*$ is a product of
 two nonconstant homogeneous polynomials, i.e.
\[a + b^* =  \Big( \sum_{\substack{u \in \axs\\ \deg(u) = d}} A_u u
\Big) \Big( \sum_{\substack{w \in \axs\\ \deg(w) = e}} B_w w\Big),\]
for some $d, e\in\N$. The coefficient of $uw$ in $a + b^*$ is
$A_uB_w$.
We see that $a + b^*$ has only terms which are analytic or
antianalytic.  Suppose $u_0,w_0$ are such that ${u_0}{w_0}$ is
analytic.  Then $u_0, w_0$ are analytic.  This implies that for any other $u$,
either $A_u = 0$ or $uw_0$ is analytic or antianalytic, which necessarily
implies that $u$
is analytic.  Similarly, either $B_w = 0$ or $w$ is analytic.  This implies
that $a + b^*$ is analytic, which is a contradiction.
\end{proof}

Now we prove 
Theorem \ref{thm:principalI}  \eqref{it:pluriharm}. 
First we recall its statement. 

\noindent
{\bf
Theorem  \ref{thm:principalI}  \eqref{it:pluriharm}}. \
 {\it
 Suppose $I \subset \FA$ is the left ideal generated by a nonconstant polynomial
$p = a +
b^*$, where $a$ and $b$ are analytic polynomials.
Then $I$ is not real if and only if $p = a - a^* + c$ for some nonzero constant
$c$.
}

\begin{proof}
First, if $p = a- a^* + c$, then $p + p^* = 2c \in I + I^*$, 
which implies that
$1 \in \sq{I}$.  Therefore $I$ is not real.

Conversely, suppose $I$ is not real.
 Let $d = \deg(p)$.  We have $\deg(p) = \max\{\deg(a),
\deg(b) \}$ since the leading polynomials of $a$ and $b^*$ are analytic and
antianalytic respectively, and hence cannot cancel each other out.
Let $a', b'$ be the degree $d$ terms of $a$ and $b$ respectively, so that
$p' = a' + (b')^*$ is the leading polynomial of $p$.

If $b' = 0$, then $p' = a'$.  By Proposition \ref{prop:leadingFactorSquares},
$p' = p_1 \cdots p_j f$ for some $f$ and some nonconstant irreducible factors
$p_i$, and
\[ \pm (p_1 \cdots p_j + p_j^* \cdots p_1^*)\]
is a sum of squares.  However, since $p'$ is analytic, this implies that $p_1
\cdots p_j$ is analytic.  An analytic polynomial has no terms of the form
$m^*m$, whereas a nonconstant sum of squares does. 
This is a contradiction, so $b' \neq 0$.  Similar reasoning shows $a' \neq 0$.

Next, by Lemma \ref{lem:antiHomIrr}, $a' + (b')^*$ is irreducible.
By Proposition \ref{prop:leadingFactorSquares},
\[ a' + b' + (a')^* + (b')^*\]
is a sum of squares.
A sum of an analytic and an antianalytic
 polynomial has no terms of the form
$m^*m$, whereas a nonconstant sum of squares does.
Therefore $a' + b' + (a')^* + (b')^* = 0$, which implies that $a' = -b'$. 

Next, since $I$ is not real, there exists a $q$ such that $\deg(q) < \deg(p)$
and $qp + p^*q^*$ is a nonzero sum of squares.
Let $q'$ be the leading polynomial of $q$.
Then
\[ q'\big(a' - (a')^*\big) + \big(a' - (a')^*\big)^*(q')^*\]
is either $0$ or is the leading polynomial of $qp + p^*q^*$.
In either case, it is a sum of squares.
Since $\deg(q) < \deg(p) = \deg\big(a' - (a')^*\big)$, by Lemma \ref{lem:factorSOS},
this implies that
$(q')^*$ is a factor of $a' -(a')^*$ on the right.
However, by Lemma \ref{lem:antiHomIrr}, $a - a^*$ is irreducible.
Therefore $q'$ is constant, which implies that $q$ is constant.
Therefore $\pm(p + p^*)$ is a nonzero sum of squares.

Let $e = \deg(p + p^*)$.  Let $a_e', b_e'$ be the degree $e$ elements of $a$
and $b$ respectively so that $a_e' + (b_e')^*$ is the leading polynomial of $p
+ p^*$.  As before, a nonconstant sum of squares cannot be expressed in this
form.  Hence $e = 0$, and so $p$ is of the form $a - a^* + c$, where $c = b_e$
is a nonzero constant. 
\end{proof}

\begin{example}
Theorem {\rm \ref{thm:principalI} \eqref{it:pluriharm}}
only applies to principal left
ideals. Indeed, if $I$ is generated by $x_1^2$ and $x_1 + x_1^*$,
then $I$ is not real.
We have $-x_1^2 + x_1(x_1 + x_1^*) = x_1x_1^* \in I + I^*$ but $x_1 \not\in
I$.
\end{example}

\begin{cor}\label{cor:linear}
 Suppose $I \subset \FA$ is the left ideal generated by
a polynomial $p$ with $\deg(p) = 1$.  Then $I$ is real if and only if
$p$ is not of the form $a - a^* + c$, where $a$ is analytic and
homogeneous of degree $1$ and $c \neq 0$ is a constant.
\end{cor}

\begin{proof}
Every polynomial $p$ of degree $1$ is of the form $p = a + b^*$, where $a, b$
are analytic, and $a$ has no constant term.
The result follows from
Theorem \ref{thm:principalI} item \ref{it:pluriharm}.
\end{proof}

\subsection{Algorithm for checking whether  a principal ideal is real}
\label{sec:poor}

The paper \cite{CHMN} gives an algorithm  for 
computing the real radical of any finitely generated left ideal $I$. 
Here we discuss an improvement of a more limited algorithm
which determines whether or not
a given left ideal 
$I$ is real. 
A test version
has been implemented, and we give
an example.

Suppose $p\in\FA_d$ is given and let $I$ be the left ideal generated by $p$.
We employ Proposition \ref{prop:degBoundsPI} to test reality of $I$. Consider
the feasibility problem
\beq\label{eq:lmi}
\begin{split}
s&\in\FA  \quad\text{is a nonzero sum of squares}\quad \\
\text{s.t. }
s & = qp+p^*q^* \quad\text{for some $q$ with }
\deg(q)<\deg(p).
\end{split}
\eeq
This is an instance of an LMI. Namely, let $M_{< d}$ be a vector whose entries
are all monomials of degree $<d$. Then $s$ is a sum of squares if and only if
there is a $G\succeq0$ with $M_{<d}^*GM_{<d}=s$;
cf.~\cite{MP,KP,Hel}. The equation $s=qp+p^*q^*$
translates into a system of linear equations involving the coefficients of $q$ and
entries of $G$. Thus \eqref{eq:lmi} is a feasibility 
semidefinite program (SDP)
\beq\label{eq:sdp}
\begin{split}
\text{Find } 0\neq G &\succeq0\\
\text{s.t. } M_{<d}^*GM_{<d} &= qp+p^*q^*, \quad \deg(q)<d
\end{split}
\eeq
and 
can thus be solved using a standard
SDP solver. To search for a nonzero $G$ we normalize \eqref{eq:sdp} by requiring 
$\tr(G)=1$.
The left ideal $I$ is non-real if and only if \eqref{eq:sdp} is feasible.

\begin{rem}
We remark that checking if a noncommutative polynomial
is a sum of squares
can be done 
exactly using quantifier elimination  $($cf.~\cite{PW98}$)$,
but this is only viable for problems of small
size
$($since the complexity for quantifier elimination is doubly exponential \cite[Section 11]{BPR}$)$. 
Hence in practice we employ SDPs 
$($which typically run in polynomial time, cf.~\cite{WSV00}$)$
for numerical verification; cf.~\href{http://ncsostools.fis.unm.si}{\tt NCSOStools} \cite{CKP} for a  computer
algebra package which does this.
\end{rem}

We demonstrate the above with a  simple example.
 
\begin{ex}
Let $p=xx^*-x^*x-1$. The corresponding left ideal $I$ is real. Indeed,
write $q=q_0+q_1x+q_2x^*$, and $G=[g_{ij}]_{i,j=1}^3$. Then \eqref{eq:sdp}
becomes
\bes\label{eq:quad}
\begin{array}{c}
G\succeq0, \\[.1cm]
 g_{1 1}+g_{2 2}+g_{3 3}-1=0 \qquad  g_{1 1}+2 q_{0}=0 \qquad  g_{1 1}+g_{2 2}+4 q_{0}+2 q_{2}=0 \\
 g_{1 1}-2 g_{1 2}-2 g_{1 3}+g_{2 2}+2 g_{2 3}+2 g_{3 3}+2 q_{1}=0 \\
 g_{1 1}+g_{1 2}+g_{1 3}-g_{2 2}-3 g_{2 3}-2 g_{3 3}+4 q_{0}-q_{1}-3 q_{2} =0\\
 g_{1 1}-4 g_{1 2}-4 g_{1 3}+5 g_{2 2}+8 g_{2 3}+4 g_{3 3}+4 q_{0}-8 q_{1}-6 q_{2} =0\\
 g_{1 1}+2 g_{1 2}+2 g_{1 3}+2 g_{2 2}+2 g_{3 3}+2 q_{0}+2 q_{1}+2 q_{2} =0\\
 g_{1 1}-2 g_{1 2}-2 g_{1 3}+g_{2 2}+2 g_{2 3}+g_{3 3}+2 q_{0}-2 q_{1}-2 q_{2}=0. \\
\end{array}
\ees
Solving the above linear system yields
\bes\label{eq:solve}
g_{1 1}=1,\quad g_{1 3}=-g_{1 2},\quad g_{2 2}=1,\quad g_{2 3}=0,\quad g_{3 3}=-1,\quad
   q_{0}=-\frac{1}{2},\quad q_{1}=0,\quad q_{2}=0.
\ees
Putting this solution back into $G$ leads to the LMI
\[
G=
\left[
\begin{array}{ccc}
 1 & g_{1 2} & -g_{1 2} \\
g_{12}& 1 & 0 \\
 -g_{12} & 0 & -1
\end{array}
\right] \succeq0,
\]
which is clearly infeasible.
\end{ex}

The algorithm based on \eqref{eq:sdp} extends to non-principal left ideals $I$ in a straightforward way.
To check for membership in $I+I^*$ (needed to encode the linear constraints in \eqref{eq:sdp}) 
any generators $p_j$ for $I$ will do, the fewer the better.
Finding a smallest generating set
for an ideal is hard, hence a reduced Gr\"obner basis is a reasonable choice.
We refer the reader to
\cite{Gr00,Mor94,Rei95,Kel97,Hey01,Lev05} and the references therein
for the theory of noncommutative Gr\"obner bases.

We have implemented the algorithm for test purposes under \href{http://math.ucsd.edu/~ncalg}{\tt NCAlgebra} \cite{HOSM}. A major limitation is construction of the LMI
required. This arises from  manipulation of the large number of monomials in $M_{<d}^*GM_{<d}$.
One could improve performance  by storing these in advance.
Also, the benefits when the generators $p_j$ of $I$ have few terms  could be explored further.

\section{Principal ideals generated by univariate quadratic nc polynomials}
\label{sec:quad}

 In Subsection \ref{subsec:further} we 
 characterized linear polynomials giving
 rise to real principal left ideals. In this section we
 discuss a complete classification whether or not the  
 left ideal generated by a univariate quadratic polynomial is real. 
 Of course we hope that an elegant general result will emerge, and
 perhaps the issues and structures exposed in this particular
 nontrivial  case 
 will provide some guidance in that direction.

\subsection{A sum of squares tool}
We start by characterizing univariate quadratics which are sums of squares,
since we need this for our classification. Its straightforward proof is left 
as an exercise for the reader.

\begin{lem}\label{lem:2sos}
For a symmetric univariate quadratic polynomial
\[
p=a_0+a_1 (x+x^*)+a_2\big(x^2+(x^*)^2\big) + a_3xx^*+a_4x^*x,
\]
the following are equivalent:
\begin{multicols}{2}
\ben[\rm(i)]
\item
$p$ is a sum of two squares;
\item
$p$ is a sum of squares;
\item
the following LMI is feasible
 \beq\label{eq:gram21}
 G= \begin{bmatrix}
  a_0 & \lambda & a_1 - \lambda \\
 \lambda  &  a_3 & a_2 \\
a_1- \lambda &  a_2& a_4
\end{bmatrix}\succeq0;
\eeq
\item
$p(X)\succeq0$ for all $n\in\N$ and all $X\in \R^{n\times n}$;
\item
$p(X)\succeq0$ for all  $X\in \R^{2\times 2}$;
\item
$-a_1^2 + a_0 (2 a_2+a_3+a_4)\geq0$,  $a_0\geq0$, and
\beq\label{gram:20}
\begin{bmatrix}
 a_3 & a_2 \\
 a_2& a_4
\end{bmatrix} \succeq0.
\eeq
\een
\end{multicols}
\end{lem}

The quantitative strengthening of the sum of squares theorem 
tells us that a $p$ as in Lemma {\rm\ref{lem:2sos}}
is a sum of squares iff it is positive semidefinite for all $X\in\R^{3\times3}$ $($cf.~\cite{Hel,MP,McC}$)$. Item {\rm (v)}
above improves this size bound a bit.

\subsection{The $p$ which generate real ideals}

Let $p$ be an arbitrary univariate and quadratic polynomial, and let $I$ be the principal left ideal it generates. 
We would like to determine when $I$ is real. 
  
\begin{prop}  
Let 
\beq\label{eq:p2}
p=a_0+a_1x+a_2x^*+a_3x^2+a_4xx^*+a_5x^*x+a_6 (x^*)^2\in\FA
\eeq
be an arbitrary quadratic univariate polynomial $($i.e., at least one of $a_3,a_4,a_5,a_6$ is nonzero$)$.
Then the left principal ideal generated by $p$ is non-real if and only if either $(1)$ or $(2)$ holds.
\begin{enumerate}[\rm (1)]
\item Either $a_4+a_6$ or $a_3+a_5$ is nonzero and there exists an integer $t$ such that
 \bes
 (-1)^t \left[ \begin{array}{cc}  2a_4 & a_3+a_6 \\ a_3+a_6 & 2a_5  \end{array}\right] \succeq 0 ,
\; (-1)^t a_0 \ge 0 \;
 \text{ \rm and }\; -(a_1+a_2)^2+4a_0(a_3+a_4+a_5+a_6) \ge 0.
 \ees
\item  $a_3+a_5=0=a_4+a_6$ and either 
\bes
a_3 + a_4 = 0 \quad\text{\rm and}\quad a_1 + a_2 = 0\quad\text{\rm and}\quad 
\big(a_2 a_4 \ne 0\; \text{ \rm or }\; (a_0 a_4 \ge 0  \text{ \rm and } a_0^2+a_4^2 \ne 0)\big)
\ees
or
\bes
a_4 + a_3 \ne 0\quad\text{\rm and}\quad a_1 + a_2 \ne 0\quad\text{\rm and}\quad a_0 =\frac{ (a_1 + a_2) (a_1 a_4-a_2 a_3)}{(a_3 + a_4)^2}.
\ees
\end{enumerate}
\end{prop}

\begin{proof}
The proof illustrates the theory in Section \ref{sec:hom} and requires calculations stemming from that.
These we found tricky, so we include them.
 By Proposition \ref{prop:degBoundsPI}, $I$ is non-real if and only if there is a linear 
 \beq\label{eq:p1}
 q=q_0+q_1x+q_2x^*\in\FA
 \eeq such that
$qp+p^*q^*$ is a nonzero sum of squares.
Compute $qp+p^*q^*$ and verify that it equals
\begin{multline}\label{eq:bla}
2 a_0  q_0 
+ x \big(a_0  q_1  + a_0  q_2 + a_1 q_0  + a_2  q_0  \big)
 + x^* \big(a_0  q_1   + a_0 q_2    + a_1  q_0    +a_2 q_0  \big) \\ 
+
x^2 \big( a_1  q_1 +a_2  q_2  +a_3  q_0 +a_6  q_0  \big)+
(x^*)^2 \big(a_1  q_1 +a_2  q_2 +a_3 q_0    +a_6 q_0    \big) \\
+2 x^*x \big( a_1 q_2   +  a_5  q_0 \big) +
+2 xx^* \big( a_2  q_1 + a_4  q_0 \big)   \\ 
+ x^3 \big(a_3 q_1  +a_6  q_2  \big) +
(x^*)^3 \big(a_3  q_1 +a_6  q_2 \big)
+ x^2 x^* \big( a_6  q_1  +a_4  q_1 \big)
+ x (x^*)^2 \big( a_6  q_1   +a_4  q_1  \big)\\
+ xx^*x \big( a_4  q_2   +a_5  q_1  \big) 
+ x^* x^2 \big(  a_3  q_2  +   a_5  q_2   \big) 
+ (x^*)^2 x \big(  a_3  q_2 +a_5 q_2  \big) 
+ x^*xx^* \big( a_4 q_2 +a_5 q_1 \big).
\end{multline}
Here line 1 of \eqref{eq:bla} collects degree 0 and 1 terms of $qp+p^*q^*$,
lines 2 and 3 contain degree 2 terms, while lines 4 and 5 contain degree 3 terms.

If \eqref{eq:bla} is a sum of squares, all degree 3 terms must vanish
(and we are thrown into the case of Lemma \ref{lem:2sos}). 
All  degree 3 terms vanishing is equivalent to the
following system of equations:
\beq\label{eq:sys}
\begin{array}{rclcrcl}
a_3 q_1  +a_6  q_2 &=&0 & \qquad\qquad & (a_4+a_6)q_1 & = & 0 \\
a_4 q_2+a_5q_1&=&0 & \qquad\qquad & (a_3+a_5)q_2&=& 0.
\end{array}
\eeq

There are two cases to consider.

\ben[\rm(a)]
\item\label{cas:a}
If either $a_4+a_6\neq0$ or $a_3+a_5\neq0$, then the system \eqref{eq:sys} has only
the trivial solution $q_1=q_2=0$.  In this case $I$ is non-real iff
\[
p+p^* = 2a_0 + (a_1+a_2)x+ (a_1+a_2)x^*  + (a_3+a_6) x^2+ 
2a_4 xx^* + 2a_5x^*x +
(a_3+a_6) (x^*)^2
\]
 is a $\pm$ sum of squares. By Lemma \ref{lem:2sos}(vi), this is true if either
 \bes
 \left[ \begin{array}{cc}  2a_4 & a_3+a_6 \\ a_3+a_6 & 2a_5  \end{array}\right] \succeq 0 \text{ and }
 -(a_1+a_2)^2+4a_0(a_3+a_4+a_5+a_6) \ge 0 \text{ and } a_0 \ge 0
 \ees
 or
\bes
 \left[ \begin{array}{cc}  2a_4 & a_3+a_6 \\ a_3+a_6 & 2a_5  \end{array}\right] \preceq 0 \text{ and }
 -(a_1+a_2)^2+4a_0(a_3+a_4+a_5+a_6) \ge 0 \text{ and } a_0 \le 0. 
 \ees

\item
If $a_3+a_5=0=a_4+a_6$ then $a_5=-a_3$ and $a_6=-a_4$, and the system \eqref{eq:sys}
reduces to the single equation 
\[
-a_3 q_1+a_4 q_2=0.
\]
If $a_3=a_4=0$, then $p$ is linear and then Corollary \ref{cor:linear} tells us when $I$ is real.
If $a_3 \ne 0$ or $a_4 \ne 0$ then there exists $t \in \RR$ such that $q_1=t a_4$ and $q_2=t a_3$.
We reevaluate $qp+p^*q^*$:
\begin{multline}\label{eq:bla2}
2 a_0 q_0 + 
x \Big( a_1 q_0  + a_2 q_0  + a_0 a_4 t + a_0 a_3 t \Big)  
+x^* \Big( a_1 q_0  + a_2 q_0  + a_0 a_4 t + a_0 a_3 t \Big)  \\
  +
  x^2 \Big(    a_3 q_0   - a_4 q_0   + a_1 a_4 t  + a_2 a_3 t \Big)
  + (x^*)^2 \Big(    a_3 q_0   - a_4 q_0   + a_1 a_4 t   + a_2 a_3 t \Big) \\
  + 2 xx^* \big( a_4 q_0  +  a_2 a_4 t \big) + 2 x^* x \Big( 
 -  a_3 q_0  + a_1 a_3 t \Big)
\end{multline}
Then $I$ is non-real if and only if there are $q_0,t\in\R$ making \eqref{eq:bla2} a nonzero sum of squares.
By assertion (vi) of Lemma \ref{lem:2sos},  \eqref{eq:bla2} is a sum of squares 
if and only if the system of the following polynomial inequalities is feasible:
\begin{equation}
\label{eq:bla3}
\begin{split}
 a_4 (a_2 t + q_0) & \geq 0 \\
 a_3 (a_1 t - q_0) & \ge 0 \\
 (a_3+a_4) q_0 & =(a_1 a_4-a_2 a_3) t \\
 (a_1+a_2) q_0 & = a_0 (a_3+a_4) t \\
 a_0 q_0 & \geq 0 \\
\end{split}
\end{equation}
We also require that \eqref{eq:bla2} has at least one nonzero coefficient. 

If the linear system (for $q_0,t$) of the third and the fourth equation in
\eqref{eq:bla3}
has a nonzero determinant, then $q_0=t=0$ which implies a contradiction that all coefficients of \eqref{eq:bla2} are zero. Therefore
\begin{equation}
\label{eq:bla4}
(a_1 + a_2) (a_1 a_4-a_2 a_3)=a_0(a_3 + a_4)^2.
\end{equation}

If $a_3+a_4=0$, then by \eqref{eq:bla4} also $a_1+a_2=0$ and \eqref{eq:bla3} simplifies considerably. 
The inequality in the second line of \eqref{eq:bla3} is now equivalent to the first, 
so we are left with the system of the first and the fifth inequality.
Moreover, \eqref{eq:bla2} has a nonzero coefficient iff   $a_4 (a_2 t + q_0) \ne 0$ or $a_0 q_0 \ne 0$. 
This system has a solution iff either $a_2 a_4 \ne 0$ or $a_0 a_4 \ge 0$ and $a_0^2+a_4^2 \ne 0$.

If  $a_3+a_4 \ne 0$ we can compute $a_0$ from \eqref{eq:bla4} and $q_0$ from the third equation of  \eqref{eq:bla3}.
In this case  the system \eqref{eq:bla3} is equivalent to $\displaystyle\frac{(a_1 + a_2) t}{a_3 + a_4} \ge 0$.
Moreover, \eqref{eq:bla2} has a nonzero coefficient iff $(a_1+a_2) t \ne 0$. This system has a solution iff $a_1+a_2 \ne 0$.

Therefore, in case (b), the ideal $I$ is non-real iff either 
\bes
a_3 + a_4 = 0 \quad\text{and}\quad a_1 + a_2 = 0\quad\text{and}\quad 
\big(a_2 a_4 \ne 0\; \text{ or }\; (a_0 a_4 \ge 0  \text{ and } a_0^2+a_4^2 \ne 0)\big)
\ees
or 
\bes
a_4 + a_3 \ne 0\quad\text{and}\quad a_1 + a_2 \ne 0\quad\text{and}\quad a_0 =\frac{ (a_1 + a_2) (a_1 a_4-a_2 a_3)}{(a_3 + a_4)^2}.
\qedhere
\ees
\een
\end{proof}

%\end{document}

\newpage

\appendix

\section{NOT FOR PUBLICATION}

In this section we collect some supplementary material that was not included for publication.

\subsection{Proof of Lemma {\rm \ref{lem:2sos}}}
It is clear that (i) implies (ii). 
By noting that $G$  in \eqref{eq:gram21} is the general form for the Gram matrix of $p$,
we can deduce that items (ii) and (iii) are equivalent.
Now suppose (iii) holds. By choosing the smallest $\la$ making $G=G(\la)\succeq0$, the rank of the corresponding $G$ is $\leq2$. As in the proof of Lemma \ref{lem:homsos} this implies $p$ is a sum of $\leq2$ squares,
so (i) holds. By the sum of squares theorem \cite{Hel}, (ii) and (iv) are equivalent. 
All this shows that (i), (ii), (iii) and (iv) are equivalent.

Clearly, (iv) implies (v). Let us now establish (v) $\Rightarrow$ (vi). First of all, $a_0=p(0)\geq0$.
Since $p$ is positive semidefinite on $\R^{2\times 2}$, its homogeneous degree $2$ part, 
$\hat p$, is also positive semidefinite on $2\times 2$ matrices. It follows that
\beq\label{eq:1}
\tr \left( \hat p \Big (
\begin{bmatrix}
0 & 1 \\ 0 & 0 
\end{bmatrix}
\Big)\right) = a_3+a_4\geq0,
\eeq
and
\beq\label{eq:2}
\det \left( \hat p \Big (
\begin{bmatrix}
0 & 1 \\ 0 & c 
\end{bmatrix}
\Big)\right) = 
a_3a_4 +c^2 (a_3a_4-a_2^2) \geq 0 \quad \text{for all } c\in\R.
\eeq From 
\eqref{eq:2} we immediately obtain
\beq\label{eq:3}
a_3a_4-a_2^2\geq0.
\eeq
Together \eqref{eq:1}  and \eqref{eq:3} show that the matrix \eqref{gram:20} has nonnegative trace and nonnegative determinant, so is positive semidefinite.
In particular, 
\[
2a_2+a_3+a_4= 
\left\langle 
\begin{bmatrix}
 a_3 & a_2 \\
 a_2& a_4
\end{bmatrix}
,
\begin{bmatrix}
 1&1\\
 1&1
\end{bmatrix}
\right\rangle\geq0.
\]
If $2a_2+a_3+a_4=0$, then $p(c)=a_0+2c a_1\geq0 $ for all $c\in\R$, whence $a_1=0$.
In this case $-a_1^2 + a_0 (2 a_2+a_3+a_4)=0$.
If $2a_2+a_3+a_4>0$, then
\[
p\left(-\frac{a_1}{2a_2+a_3+a_4}\right) =
\frac{-a_1^2 + a_0 (2 a_2+a_3+a_4)}{2a_2+a_3+a_4} 
\geq0\]
shows $-a_1^2 + a_0 (2 a_2+a_3+a_4)\geq0$, as desired.

Finally, assume (vi) holds. If $2 a_2+a_3+a_4=0$, then 
$a_1=0$, and hence $p$ has no linear terms. Since \eqref{gram:20} is
positive semidefinite, the homogeneous quadratic part of $p$ is a sum of squares. As also $a_0\geq0$,
$p$ is a sum of squares. 
We can thus assume 
$
2a_2+a_3+a_4>0.
$
Perfoming an affine linear change of variables $x\mapsto x-
\frac{a_1}{2a_2+a_3+a_4}$ leads to the polynomial
\[
\tilde p(x) = p\left(x
-\frac{a_1}{2a_2+a_3+a_4}\right) = 
\frac{-a_1^2 + a_0 (2 a_2+a_3+a_4)}{2a_2+a_3+a_4} + 
a_2(x^2+(x^*)^2) + a_3xx^*+a_4x^*x
\]
having no linear term. 
Obviously, $ p$ is positive (resp., sum of squares) iff $\tilde p$ is positive (resp., sum of squares),
and the latter is the case iff 
its constant term is nonnegative and its homogeneous quadratic part is a sum of squares, i.e., iff
$-a_1^2 + a_0 (2 a_2+a_3+a_4)\geq0$ and \eqref{gram:20} is positive semidefinite.
\qed

\subsection{Left Gr\"obner Basis Algorithm on $\FA$}
\label{sect:leftGB}

A useful tool for the general algorithm for testing reality of non-principal left ideals 
 is computation of a left Gr\"obner basis
for an ideal. This was implicit in \cite{CHMN} but here we give
a cleaner description which lends itself to our computation.
For details we refer the reader to the extensive literature on noncommutative
Gr\"obner bases; see e.g.~\cite{Mor94,Rei95,Kel97,Hey01,Lev05} and the references therein.

Fix a monomial order $\succ$ on $\axs$.
For convenience, we choose an order such that
$\deg(u) < \deg(v)$ implies $u \prec v$.
By definition, $\succ$ satisfies the descending chain condition.

Let $p_1, \ldots, p_k$ generate a left ideal $I \subset \FA$,
and assume that they are monic (i.e. the coefficient of the leading monomial
is $1$).
Let $u_1, \ldots, u_k$ be the leading monomials of $p_1, \ldots, p_k$
respectively.
If $u_i \mid u_j$ for any $i \neq j$, let $\omega$ be a monomial such that
 $u_i = \omega u_j$.
Replace $p_j$ by $ A(p_j - \omega p_i)$, where $A$ is a normalization
making the latter polynomial monic. 
In this case, the leading monomial of this new
$p_j$ is lower than the leading monomial of the old $p_j$.
If any of the new $p_j$ are $0$, remove them from our set.

Repeat this until $u_i \nmid u_j$ for any $i \neq j$.
Then $\{p_1, \ldots, p_k\}$ is a left Gr\"obner basis.  
This Algorithm is guaranteed to 
terminate since at each step where the algorithm does not stop,
we replace a polynomial with another polynomial whose leading monomial is lower
than that of the polynomial being replaced.

\begin{prop}
 Let $p_1, \ldots, p_k$ be a left Gr\"{o}bner basis for a
left ideal $I$, and suppose $\deg(p_i) \leq d$ for each $i$.
Then for each $e \geq d$, a basis for $I_e$ is
\begin{equation}
\label{eq:leftGBBasis}
 \{ vp_i\mid 1 \leq i \leq k,\ v \in \axs_{e-\deg(p_i)}\}.
\end{equation}
\end{prop}

\begin{proof}
 Any polynomial $\iota \in I$ is equal to
\[ \iota = q_1p_1 + \ldots + q_kp_k\]
for some polynomials $q_i$ since the $p_i$ generate $I$.
Let $\omega_i$ be the leading monomial of each $q_i$,
and let $u_i$ be the leading monomial of each $p_i$ so that the leading
monomial of $q_ip_i$ is $\omega_iu_i$.

If $i \neq j$, then $\omega_iu_i \neq \omega_ju_j$; indeed,
suppose $\omega_iu_i = \omega_ju_j$, and suppose $\deg(u_i) \leq \deg(u_j)$.
Then $\omega_iu_i = \omega_ju_j$ implies that $u_i \mid u_j$, which is a
contradiction by the construction of the left Gr\"obner basis.

Therefore none of the leading monomials of the $q_ip_i$ cancel each other out.
In particular, $\deg(\iota) = \max\{ \deg(q_ip_i) \}$.
This shows that the set \eqref{eq:leftGBBasis} spans $I_e$ for each $e \geq d$.
Likewise, since the leading monomials of $v_ip_i$ cannot cancel each other out,
the set \eqref{eq:leftGBBasis}  is linearly independent.
\end{proof}

\subsection{The Real Ideal Algorithm for finitely generated left ideals}
Suppose $I\subseteq\FA$ is a finitely generated left ideal
with left Gr\"obner basis
$\{p_1,\ldots,p_r\}\subseteq\FA_d$.
Consider
the feasibility problem
\beq\label{eq:lmi2}
\begin{split}
s&\in\FA  \quad\text{is a nonzero sum of squares}\quad \\
\text{s.t. }
s & = \sum_{j=1}^r \big(q_jp_j+p_j^*q_j^*\big), \quad
\deg(q_j p_j)<2 d.
\end{split}
\eeq
As in Subsection \ref{sec:poor},
this is an instance of an LMI,
\beq\label{eq:sdp2}
\begin{split}
\text{Find } 0\neq G &\succeq0\\
\text{s.t. } M_{<d}^*GM_{<d} &= \sum_{j=1}^r \big(q_jp_j+p_j^*q_j^*\big), \quad
\deg(q_j p_j)<2 d.
\end{split}
\eeq
and 
can thus be solved using a standard
SDP solver. To search for a nonzero $G$ we normalize \eqref{eq:sdp2} by requiring
$\tr(G)=1$.
The ideal $I$ is real if and only if \eqref{eq:sdp2} is infeasible.

\begin{example}
Here is a simple univariate example.
Let $\succ$ be the following monomial order on $\axs$: 
for $u,w
\in \axs$, we define  $u \prec w$ if $\deg(u) < \deg(w)$ or if $\deg(u) = \deg(w)$ and $u = r
x^* s$ and $w = r x t$ for some $r,s,t \in \axs$.

Let the following set generate a left ideal
\[ S=\{ x^3 + 1, x^2 + (x^*)^2,  xx^* - (x^*)^2, x^*x - 5\}.\]
Observe that these polynomials are presented in an $\succ$-decreasing manner.
We see $x^2$ divides $x^3$, so we replace $x^3 + 1$ in $S$ with
\[-\big( x^3 + 1 - x(x^2 + (x^*)^2)\big) = x(x^*)^2 - 1.\]
We thus obtain the generating set
\begin{equation}
 \label{eq:finalLeftGB}
S'=\{ x(x^*)^2 - 1, x^2 + (x^*)^2, xx^* - (x^*)^2, x^*x -
5\}.
\end{equation}
Clearly, $S'$ is a left Gr\"obner basis. We used a Mathematica implementation of
\eqref{eq:sdp2} based on \href{http://math.ucsd.edu/~ncalg}{\tt NCAlgebra} \cite{HOSM} to verify that the left ideal generated by $S'$ is  real.
\end{example}

To construct \eqref{eq:lmi2} or \eqref{eq:sdp2}, 
any generators $p_j$ for $I$ will do, the fewer the better. Finding a smallest generating set
for an ideal is hard, hence a reduced Gr\"obner basis is a reasonable choice.
The left Gr\"ober basis algorithm is useful in producing a fairly small basis for the ideal $I$. 

We have implemented the algorithm for test purposes. A major limitation is construction of the LMI
required. This arises from  manipulation of the large number of monomials in $M_{<d}^*GM_{<d}$.
One could improve performance  by storing these in advance.
Also, the benefits when the $p_j$ have few terms  could be explored further.

\end{document}